\documentclass{article}%
\usepackage{amsfonts}
\usepackage{amsmath}
\usepackage{amssymb}
\usepackage{graphicx}%
\providecommand{\U}[1]{\protect\rule{.1in}{.1in}}
\newtheorem{theorem}{Theorem}

\newtheorem{corollary}[theorem]{Corollary}

\newtheorem{definition}[theorem]{Definition}

\newtheorem{lemma}[theorem]{Lemma}

\newenvironment{proof}[1][Proof]{\noindent\textbf{#1.} }{\ \rule{0.5em}{0.5em}}
\begin{document}

\title{Topological versions of Abel-Jacobi, the height pairing, and the Poincar\'{e} bundle}
\author{Mirel Caib\u{a}r
\and C. Herbert Clemens}
\date{November 25, 2012}
\maketitle

\section{Introduction}

This note continues the extension of normal function constructions to the
topological setting begun in \cite{C}. This time we begin with a topological
version of the Abel-Jacobi map for algebraic cycles, a slight variant of the
topological version developed in \cite{HLZ} and \cite{DK}. We then go on to
construct a topological version of the height pairing for cycles and interpret
it as a lifting of (the Abel-Jacobi image of) the pair of cycles to a point in
the fiber of the Poincar\'{e} bundle. The construction was inspired by the
work of Hain \cite{Hai} in the algebraic case. In fact, a one-sentence summary
is that we simply unwind Hain's constructions on a principally polarized
(intermediate) Jacobian bundle from the underlying complex structure via the
Cheeger-Simons theory of differential characters.

In the algebraic setting, the archimedian height pairing of Beilinson, Bloch,
Gillet-Soul\'{e} is a kind of linking number between two disjoint,
homologically trivial cycles of codimension $n$ on a $\left(  2n-1\right)
$-dimensional algebraic manifold. (See, for example, Definition 1.1 of
\cite{Wa}.) We adapt this notion to topological cycles in a way that makes
clearer the behaviour of the height pairing as one of the cycles is moved in
such a way as to cross the other. Roughly speaking, given two primitive
classes $\eta,\eta^{\prime}\in H^{2n}\left(  W;\mathbb{Z}\right)  $ on a
complex manifold $W$ of dimension $2n$ and their associated normal functions%
\[
\alpha_{\eta},\alpha_{\eta^{\prime}}:\mathbb{L}\rightarrow J\left(
X_{\mathbb{L}}/\mathbb{L}\right)
\]
over a Lefschetz pencil $\mathbb{L}$ as in \cite{C}, in this note we associate
to each $p\in\mathbb{L}$ submanifolds $\Sigma_{p},\Sigma_{p}^{\prime}\subseteq
X_{p}$ of real dimension $2\left(  n-1\right)  $ whose respective loci, as $p$
transverses $\mathbb{L}$ cut out the Poincar\'{e} duals of $\eta$ and
$\eta^{\prime}$ respectively. Let $\beta_{\eta^{\prime}}$ denote the image of
$\alpha_{\eta^{\prime}}$ under the isomorphism%
\[
J\left(  X_{\mathbb{L}}/\mathbb{L}\right)  \rightarrow J\left(  X_{\mathbb{L}%
}/\mathbb{L}\right)  ^{\vee}%
\]
induced by the principal polarization. If $%
{\displaystyle\int\nolimits_{X_{p}}}
\eta\wedge\eta^{\prime}=0$ the loci $\left\{  \Sigma_{p}\right\}
_{p\in\mathbb{L}}$ and $\left\{  \Sigma_{p}^{\prime}\right\}  _{p\in
\mathbb{L}}$ can be chosen to be disjoint. By the results explained in this
note, this corresponds to the fact that we have a global lifting of the graph
of%
\[
\left(  \alpha_{\eta},\beta_{\eta^{\prime}}\right)  :\mathbb{L}\rightarrow
J\left(  X_{\mathbb{L}}\right)  \times_{\mathbb{L}}J\left(  X_{\mathbb{L}%
}\right)  ^{\vee}/\mathbb{L}%
\]
to a never-zero section of the pull-back of the Poincar\'{e} bundle,
indicating that the Chern class of the pull-back is zero, as it must be by
\cite{C}. On the other hand, if $%
{\displaystyle\int\nolimits_{X_{p}}}
\eta\wedge\eta^{\prime}=m$, then the obstructions to keeping the loci the loci
$\left\{  \Sigma_{p}\right\}  _{p\in\mathbb{L}}$ and $\left\{  \Sigma
_{p}^{\prime}\right\}  _{p\in\mathbb{L}}$ disjoint can be localized to simple
obstructions at $m$ points $p_{j}\in\mathbb{L}$ such that the $X_{p_{j}}$ are
smooth. Each contributes winding number $1$ to the intersection number of
$\eta$ and $\eta^{\prime}$.

We then go on to a more detailed analysis of the case in which $\Sigma$ and
$\Sigma^{\prime}$ lie in $CH_{hom}^{n}\left(  X_{p}\right)  $. We show that
the topological height pairing is simply the real part of a complex invariant
whose imaginary part $\left[  \Sigma.\Sigma^{\prime}\right]  $ is the height
pairing defined in \cite{GS}. We go on to derive some corollaries of this fact.

The purpose the Part I of this paper is to define a topological height pairing
$\left(  \Sigma.\Sigma^{\prime}\right)  $ for homologically trivial cycles on
a fixed smooth $X_{p}=X$ and to show that this height pairing is nothing more
nor less than the choice of a section of a topological version of the
Poincar\'{e} bundle considered as a circle bundle over $J\left(  X\right)
\times J\left(  X\right)  ^{\vee}$. The purpose of Part II of this paper is to
treat the case in which $\Sigma$ and $\Sigma^{\prime}$ lie in $CH_{hom}%
^{n}\left(  X\right)  $. We show that
\[
\left(  \Sigma.\Sigma^{\prime}\left(  t\right)  \right)  +i\left[
\Sigma.\Sigma^{\prime}\left(  t\right)  \right]
\]
varies holomorphically for an algebraic family $\left\{  \Sigma^{\prime
}\left(  t\right)  \right\}  _{t\in T}$ and characterize it as a section of
the flat line bundle on $T$ given by the incidence divisor of $\Sigma$.

The authors are extremely grateful to Christian Schnell, who, upon receipt of
\cite{C}, pointed us in the direction of Hain's paper and to referee of an
earlier version for pointing out inaccuracies.

Two notational remarks--all integral homology and cohomology in this paper is
taken modulo torsion and, for a cycle $\Sigma$ on $X$, $\left\vert
\Sigma\right\vert \subseteq X$ will denote the support of the cycle.

\part{Topological setting}

\section{The setting}

Let $X$ denote a smooth complex projective manifold of (complex) dimension
$2n-1$ with very ample polarization $\mathcal{O}_{X}\left(  1\right)  $ and
K\"{a}hler metric induced by the associated projective imbedding. Define%
\begin{align*}
J\left(  X\right)   &  =\frac{H_{2n-1}\left(  X;\mathbb{R}\right)  }%
{H_{2n-1}\left(  X;\mathbb{Z}\right)  }\\
J\left(  X\right)  ^{\vee}  &  =\frac{H^{2n-1}\left(  X;\mathbb{R}\right)
}{H^{2n-1}\left(  X;\mathbb{Z}\right)  }=Hom\left(  H_{2n-1}\left(
X;\mathbb{Z}\right)  ,\frac{\mathbb{R}}{\mathbb{Z}}\right)  .
\end{align*}

\begin{definition}
\label{PB}The Poincar\'{e} circle bundle is the bundle whose total space is
the quotient of the action of $H_{2n-1}\left(  X;\mathbb{Z}\right)  \times
H^{2n-1}\left(  X;\mathbb{Z}\right)  $ on%
\[
H_{2n-1}\left(  X;\mathbb{R}\right)  \times H^{2n-1}\left(  X;\mathbb{R}%
\right)  \times\frac{\mathbb{R}}{\mathbb{Z}}%
\]
by the rule that $\left(  m,n\right)  \in H_{2n-1}\left(  X;\mathbb{Z}\right)
\times H^{2n-1}\left(  X;\mathbb{Z}\right)  $ acts by%
\[
\left(  \alpha,\beta,r\right)  \mapsto\left(  \alpha+m,\beta+n,r+n\left(
\alpha\right)  +\beta\left(  m\right)  \right)  .
\]
Thus%
\[
d_{J\left(  X\right)  }d_{\left(  J\left(  X\right)  ^{\vee}\right)  }\left(
\beta\left(  \alpha\right)  \right)
\]
is the Chern class of the Poincar\'{e} bundle.
\end{definition}

The justification for the name in the above definition comes from the fact
that the restriction of the bundle to%
\[
J\left(  X\right)  \times\left\{  \beta\right\}
\]
corresponds to a $U\left(  1\right)  $-representation
\begin{gather*}
\pi_{1}\left(  J\left(  X\right)  ,0\right)  \rightarrow U\left(  1\right) \\
m\mapsto e^{2\pi i\beta\left(  m\right)  }%
\end{gather*}
and so to the flat holomorphic line bundle with the above holonomy group.

\section{Vanishing cycles, the topological Abel-Jacobi map and the height
pairing}

In this section we use without citation various standard operator identities
in \cite{We}, Chapters I and II. Let $\Sigma\subseteq X$ be any smoothly
embedded oriented submanifold of (real) dimension $2n-2$ which induces the
$0$-class in $H_{2n-2}\left(  X;\mathbb{Z}\right)  $. We call $\Sigma$ a
\textit{vanishing cycle}. For $\left\{  \Gamma\right\}  \in H_{2n-1}\left(
X,\Sigma;\mathbb{Z}\right)  $ with%
\[
\partial\Gamma=\Sigma,
\]
we consider $\Gamma$ as a $\left(  2n-1\right)  $-current. We endow $X$ with
the Riemannian metric induced from the Fubini-Study metric on $\mathbb{P}%
\left(  H^{0}\left(  \mathcal{O}_{X}\left(  1\right)  \right)  ^{\vee}\right)
$. We have that $\Delta_{d}=\Delta_{d^{C}}$ and the orthogonal decomposition%
\[
\widehat{\mathcal{A}}_{X}^{\ast}=\Delta_{d}G_{d}\widehat{\mathcal{A}}%
_{X}^{\ast}\oplus\mathcal{H}^{\ast}=\Delta_{d^{C}}G_{d^{C}}\widehat
{\mathcal{A}}_{X}^{\ast}\oplus\mathcal{H}^{\ast}%
\]
so that, as a (rectifiable) current,%
\begin{align}
\Sigma &  =dd^{\ast}G_{d}\left(  \Sigma\right) \nonumber\\
&  =dd^{\ast}G_{d}\left(  d^{C}\left(  d^{C}\right)  ^{\ast}G_{d}\left(
\Sigma\right)  +\left(  d^{C}\right)  ^{\ast}d^{C}G_{d}\left(  \Sigma\right)
\right) \label{gf}\\
&  =dd^{C}\left(  dd^{C}\right)  ^{\ast}G_{d}^{2}\left(  \Sigma\right)
-d\left(  dd^{C}\right)  ^{\ast}G_{d}^{2}\left(  d^{C}\left(  \Sigma\right)
\right)  .\nonumber
\end{align}
Notice that, if the current $\Sigma$ is $d^{C}$-closed, e.g. if $\Sigma$ is a
complex-analytic $\left(  n-1\right)  $-cycle, then%
\begin{equation}
\Sigma=dd^{C}\left(  dd^{C}\right)  ^{\ast}G_{d}^{2}\left(  \Sigma\right)  .
\label{anal}%
\end{equation}

In general we define%
\[
\omega_{\Sigma}=d^{\ast}G_{d}\left(  \Sigma\right)  .
\]
Thus%
\[
\left.  \omega_{\Sigma}\right\vert _{X-\left\vert \Sigma\right\vert }%
\]
is closed with integral periods and has integral $1$ against the boundary of a
small oriented ball transverse to $\Sigma$ and so, if $\Sigma$ is connected,
it can be thought of as the positively oriented generator of the cokernel of
the map%
\[
H^{2n-1}\left(  X;\mathbb{Z}\right)  \rightarrow H^{2n-1}\left(  X-\left\vert
\Sigma\right\vert ;\mathbb{Z}\right)  .
\]
In the general case we want the element of the cokernel represented in
$H^{0}\left(  \left\vert \Sigma\right\vert ;\mathbb{Z}\right)  $ by a cycle
consisting of a single point on each component of $\Sigma$.

Next we have the harmonic decomposition%
\begin{align}
\Gamma &  =d^{\ast}dG_{d}\left(  \Gamma\right)  +dd^{\ast}G_{d}\left(
\Gamma\right)  +\psi_{\Sigma}\label{Hodgedecomp}\\
&  =d^{\ast}G_{d}\left(  \Sigma\right)  +dd^{\ast}G_{d}\left(  \Gamma\right)
+\psi_{\Sigma}\nonumber\\
&  =\omega_{\Sigma}+dd^{\ast}G_{d}\left(  \Gamma\right)  +\psi_{\Sigma
}\nonumber
\end{align}
where $\psi_{\Sigma}\in\mathcal{H}^{2n-1}$. For any $\eta\in\mathcal{H}%
^{2n-1}$ we have%
\[%
{\displaystyle\int\nolimits_{X}}
\psi_{\Sigma}\wedge\eta=%
{\displaystyle\int\nolimits_{\Gamma}}
\eta.
\]
If, in $\left(  \ref{Hodgedecomp}\right)  $ $\Gamma$ is replaced by another
chain $\Gamma^{\prime}$ with harmonic projection $\psi_{\Sigma}^{\prime}$ and
such that $\partial\Gamma^{\prime}=\Sigma$, then $d\left(  \Gamma
-\Gamma^{\prime}\right)  =0$ so that%
\[
\Gamma-\Gamma^{\prime}=dd^{\ast}\left(  \Gamma-\Gamma^{\prime}\right)
+\psi_{\Sigma}-\psi_{\Sigma}^{\prime}%
\]
and so%
\[
\left(  \psi_{\Sigma}-\psi_{\Sigma}^{\prime}\right)  \in\mathcal{H}^{2n-1}%
\]
is an integral class.

\begin{definition}
The element $\left\{  h\mapsto%
{\displaystyle\int\nolimits_{X}}
\psi_{\Sigma}\wedge h\right\}  \in J\left(  X\right)  $ is called the
\textit{Abel-Jacobi image} of $\Sigma$, which we denote as $\alpha_{\Sigma}$.
The element $\left\{  \psi_{\Sigma}\right\}  \in J\left(  X\right)  ^{\vee}$
will be denoted as $\beta_{\Sigma}$.
\end{definition}

Notice that $\beta_{\Sigma}$ the image of $\alpha_{\Sigma}$ under the
isomorphism%
\[
J\left(  X\right)  \rightarrow J\left(  X\right)  ^{\vee}%
\]
induced by the Poincar\'{e} duality and that%
\[
\beta_{\Sigma}:H_{2n-1}\left(  X;\mathbb{Z}\right)  \rightarrow\frac
{\mathbb{R}}{\mathbb{Z}}\cong U\left(  1\right)
\]
exactly defines a flat line bundle $L_{\beta_{\Sigma}}$ on $J\left(  X\right)
$ with (flat) $U\left(  1\right)  $-connection.

Next consider the expression%
\[%
{\displaystyle\int\nolimits_{\Gamma^{\prime}}}
\omega_{\Sigma}.
\]
Fixing $\Sigma^{\prime}=\partial\Gamma^{\prime}$ the quantity $%
{\displaystyle\int\nolimits_{\Gamma^{\prime}}}
\omega_{\Sigma}$ is only well-defined modulo
\[
\left\{
{\displaystyle\int\nolimits_{\Theta^{\prime}}}
\omega_{\Sigma}:\left\{  \Theta^{\prime}\right\}  \in H_{2n-1}\left(
X-\left\vert \Sigma\right\vert ;\mathbb{Z}\right)  \right\}  .
\]
Since $d\left(  \omega_{\Sigma}\right)  =\Sigma$, the local periods around
$\left\vert \Sigma\right\vert $ are integers, hence $\left\{
{\displaystyle\int\nolimits_{\Gamma^{\prime}}}
\omega_{\Sigma}\right\}  \in\frac{\mathbb{R}}{\mathbb{Z}}$ is well-defined
modulo%
\[
G_{\Sigma}:=\left\{
{\displaystyle\int\nolimits_{\Theta^{\prime}}}
\omega_{\Sigma}:\left\{  \Theta^{\prime}\right\}  \in H_{2n-1}\left(
X;\mathbb{Z}\right)  \right\}  .
\]

\begin{definition}
\label{2}Given any two smoothly embedded oriented disjoint submanifolds
$\Sigma,\Sigma^{\prime}\subseteq X$ of (real) dimension $2n-2$ which each
induce the $0$-class in $H_{2n-2}\left(  X;\mathbb{Z}\right)  $, we define a
\textit{topological height pairing}%
\begin{equation}
\left(  \Sigma.\Sigma^{\prime}\right)  =%
{\displaystyle\int\nolimits_{\Gamma^{\prime}}}
\omega_{\Sigma}\in\frac{\mathbb{R}}{G_{\Sigma}+\mathbb{Z}}. \label{14}%
\end{equation}

\end{definition}

We extend both of the definitions just above to cycles%
\[
\Sigma=%
{\displaystyle\sum\nolimits_{i}}
m_{i}\Sigma_{i}%
\]
by linearity. In Definition \ref{2} $\omega_{\Sigma}$ is a current determined
by a smooth form on $X-\left\vert \Sigma\right\vert $. Thus we are allowed to
integrate against a chain $\Gamma^{\prime}$ whose support is disjoint from
$\left\vert \Sigma\right\vert $. In what follows, we will always assume such
disjointness. Also if we integrate the wedge product of two currents, we will
always restrict to the case in which the singular supports of the two currents
are disjoint. Compare this definition with that of \cite{Hai}. In the next two
Sections we use the height pair to define a lifting of the point $\left(
\alpha_{\Sigma},\beta_{\Sigma^{\prime}}\right)  $ to the fiber of the
Poincar\'{e} bundle above that point.

If $h$ denotes the harmonic projector, $%
{\displaystyle\int\nolimits_{X}}
\omega_{\Sigma}\wedge h\left(  \Theta^{\prime}\right)  =0$ and%
\[
\Theta^{\prime}-h\left(  \Theta^{\prime}\right)  =dd^{\ast}G_{d}\Theta
^{\prime}%
\]
is exact on $X-\left\vert \Theta^{\prime}\right\vert $. For the harmonic
decompositions%
\begin{align*}
\Gamma &  =dd^{\ast}G_{d}\left(  \Gamma\right)  +\omega_{\Sigma}+h\left(
\Gamma\right) \\
\Theta^{\prime}  &  =dd^{\ast}G_{d}\Theta^{\prime}+0+h\left(  \Theta^{\prime
}\right)
\end{align*}
we have%
\begin{align}
\Gamma\cdot\Theta^{\prime}  &  =-%
{\displaystyle\int\nolimits_{\Theta^{\prime}}}
\omega_{\Sigma}+%
{\displaystyle\int\nolimits_{\Gamma}}
h\left(  \Theta^{\prime}\right) \label{zee}\\
&  \in\mathbb{Z}\nonumber
\end{align}
as long as $\left\vert \Theta^{\prime}\right\vert $ is disjoint from
$\left\vert \Sigma\right\vert $. So $%
{\displaystyle\int\nolimits_{\Theta^{\prime}}}
\omega_{\Sigma}$ is well-defined. Indeed this is exactly B. Harris's height
pairing $\left(  \Sigma,\Theta^{\prime}\right)  $ in \cite{Har}.

To see what happens when $\Theta^{\prime}$ is moved to cross $\Gamma$
transversely, consider a one-parameter family $\Sigma^{\prime}\left(
t\right)  $ of $\Sigma^{\prime}$ parametrized by a complex parameter $t$ such
that%
\begin{equation}%
{\displaystyle\bigcup\nolimits_{\left\vert t\right\vert <\varepsilon}}
\Sigma^{\prime}\left(  t\right)  \label{inf1}%
\end{equation}
forms a smooth submanifold of $X$ fibered over the $t$-disk and meets $\Sigma$
transversally at a point of $\Sigma^{\prime}\left(  0\right)  $. Consider the
real-valued function%
\[
f\left(  t\right)  =\left(  \Sigma^{\prime}\left(  t\right)  .\Sigma\right)
=-%
{\displaystyle\int\nolimits_{\Gamma_{t}^{\prime}}}
\omega_{\Sigma}.
\]
The manifold $%
{\displaystyle\bigcup\nolimits_{\left\vert t\right\vert =\varepsilon/2}}
\Sigma^{\prime}\left(  t\right)  $ is homologous in $X-\left\vert
\Sigma\right\vert $ to the boundary of a small ball in $%
{\displaystyle\bigcup\nolimits_{\left\vert t\right\vert =\varepsilon/2}}
\Sigma^{\prime}\left(  t\right)  $ around the point $\Sigma^{\prime}\left(
0\right)  \cap\Sigma$. Since $d\omega_{\Sigma}=\Sigma$ it follows immediately
that%
\begin{equation}
f\left(  e^{2\pi i}t\right)  -f\left(  t\right)  =\pm1. \label{inf}%
\end{equation}
Thus the height pairing is well-defined as an element of $\frac{\mathbb{R}%
}{\mathbb{Z}}$.

Finally we compare $\left(  \Sigma.\Sigma^{\prime}\right)  $ and $\left(
\Sigma^{\prime}.\Sigma\right)  $, at least modulo $\mathbb{Z}$ as follows. As
in $\left(  \ref{zee}\right)  $ we have%
\begin{align}
\Gamma\cdot\Gamma^{\prime}  &  =%
{\displaystyle\int\nolimits_{\Gamma}}
\omega_{\Sigma^{\prime}}-%
{\displaystyle\int\nolimits_{\Gamma^{\prime}}}
\omega_{\Sigma}+%
{\displaystyle\int\nolimits_{\Gamma}}
h\left(  \Gamma^{\prime}\right) \label{zee'}\\
&  \in\mathbb{Z}\nonumber
\end{align}
for $\partial\Gamma=\Sigma$ and $\partial\Gamma^{\prime}=\Sigma^{\prime}$
since $\left\vert \Sigma\right\vert $ and $\left\vert \Sigma^{\prime
}\right\vert $ are assumed disjoint.

\section{Cohomology filtered by weight}

Again fix disjoint $\Sigma$ and $\Sigma^{\prime}$. In what follows we will
assume $\left\vert \Sigma\right\vert $ and $\left\vert \Sigma^{\prime
}\right\vert $ connected. The general case follows by linearity. Analogously
to \cite{Hai}, we wish to describe $\left(  \Sigma.\Sigma^{\prime}\right)  $
as the obstruction to splitting some kind of natural structure on real
cohomology. Since any $\left(  2n-1\right)  $-form restricts to zero on
$\Sigma$ and $\Sigma^{\prime}$ each harmonic form representing a class in
$H^{2n-1}\left(  X;\mathbb{Z}\right)  $ has a canonical lifting to
$H^{2n-1}\left(  X,\left\vert \Sigma\right\vert ;\mathbb{R}\right)  $ and to
$H^{2n-1}\left(  X,\left\vert \Sigma^{\prime}\right\vert ;\mathbb{R}\right)
$. Now $H^{2n-1}\left(  X-\left\vert \Sigma\right\vert ,\left\vert
\Sigma^{\prime}\right\vert ;\mathbb{R}\right)  $ is the $\left(  2n-1\right)
$-cohomology of the deRham complex%
\[
\left(  \mathcal{A}_{\left\vert \Sigma^{\prime}\right\vert }^{\ast-1}%
\oplus\mathcal{A}_{X-\left\vert \Sigma\right\vert }^{\ast},D\right)
\]
giving the mapping cone associated to the inclusion%
\[
\left\vert \Sigma^{\prime}\right\vert \rightarrow X-\left\vert \Sigma
\right\vert
\]
Since any $\left(  2n-1\right)  $-form restricts to zero on $\Sigma$, the
harmonic representatives $\mathcal{H}_{\mathbb{R}}^{2n-1}\left(  X\right)  $
for $H^{2n-1}\left(  X;\mathbb{R}\right)  $ and $\mathcal{H}_{\mathbb{R}%
}^{2n-2}\left(  \left\vert \Sigma^{\prime}\right\vert \right)  $ for
$H^{2n-2}\left(  \left\vert \Sigma^{\prime}\right\vert ;\mathbb{R}\right)  $
then allow us to define a distinguished isomorphism%
\begin{equation}
H^{2n-1}\left(  X-\left\vert \Sigma\right\vert ,\Sigma^{\prime};\mathbb{R}%
\right)  \cong\mathbb{R}\cdot\omega_{\Sigma}\oplus\mathcal{H}_{\mathbb{R}%
}^{2n-1}\left(  X\right)  \oplus\mathcal{H}_{\mathbb{R}}^{2n-2}\left(
\left\vert \Sigma^{\prime}\right\vert \right)  . \label{distiso}%
\end{equation}

Now
\[
H_{2n-1}\left(  X-\left\vert \Sigma\right\vert ,\left\vert \Sigma^{\prime
}\right\vert ;\mathbb{R}\right)  \cong H^{2n-1}\left(  X-\left\vert
\Sigma\right\vert ,\Sigma^{\prime};\mathbb{R}\right)  ^{\vee}%
\]
So via $\left(  \ref{distiso}\right)  $ we obtain
\[
H_{2n-1}\left(  X-\left\vert \Sigma\right\vert ,\left\vert \Sigma^{\prime
}\right\vert ;\mathbb{R}\right)  \cong\left(  \mathbb{R}\cdot\omega_{\Sigma
}\right)  ^{\vee}\oplus\mathcal{H}_{\mathbb{R}}^{2n-1}\left(  X\right)
^{\vee}\oplus\mathcal{H}_{\mathbb{R}}^{2n-2}\left(  \left\vert \Sigma^{\prime
}\right\vert \right)  ^{\vee}%
\]
and, writing%
\begin{align*}
M_{\mathbb{R}}  &  :=H_{2n-1}\left(  X-\left\vert \Sigma\right\vert
,\left\vert \Sigma^{\prime}\right\vert ;\mathbb{Z}\right) \\
M_{\mathbb{R}}^{\vee}  &  :=\mathbb{R}\cdot\omega_{\Sigma}\oplus
\mathcal{H}_{\mathbb{R}}^{2n-1}\left(  X\right)  \oplus\mathcal{H}%
_{\mathbb{R}}^{2n-2}\left(  \left\vert \Sigma^{\prime}\right\vert \right)  ,
\end{align*}
we have an induced injection%
\begin{gather}
M_{\mathbb{Z}}\rightarrow M_{\mathbb{R}}\label{11}\\
\left(  \partial\Delta_{\Sigma},\left\{  \Gamma_{k}\right\}  ,\Gamma^{\prime
}\right)  \mapsto\left(
{\displaystyle\int\nolimits_{\partial\Delta_{\Sigma}}}
,\left\{
{\displaystyle\int\nolimits_{\Gamma_{k}}}
\right\}  ,%
{\displaystyle\int\nolimits_{\Gamma^{\prime}}}
\right) \nonumber
\end{gather}
where $\Gamma^{\prime}$ is supported away from $\left\vert \Sigma\right\vert
\cup\left\vert \Sigma^{\prime}\right\vert $ and $\left\{  \Gamma_{k}\right\}
$ is an integral basis of $H_{2n-1}\left(  X;\mathbb{Z}\right)  $ also
supported away from $\left\vert \Sigma\right\vert \cup\left\vert
\Sigma^{\prime}\right\vert $ and $\Delta_{\Sigma}$ is a small real oriented
$2n$-ball meeting $\left\vert \Sigma\right\vert $ transversely at one point.

Finally we can filter the $\mathbb{Z}$-module $M$ via%
\begin{gather}
W_{2}=M\label{18}\\
W_{1}=\mathrm{ker}\left(  H_{2n-1}\left(  X-\left\vert \Sigma\right\vert
,\left\vert \Sigma^{\prime}\right\vert ;\mathbb{Z}\right)  \overset{\partial
}{\longrightarrow}H_{2n-2}\left(  \left\vert \Sigma^{\prime}\right\vert
;\mathbb{Z}\right)  \right) \nonumber\\
W_{0}=\mathrm{image}\left(  H_{0}\left(  \left\vert \Sigma\right\vert
;\mathbb{Z}\right)  \rightarrow H_{2n-1}\left(  X-\left\vert \Sigma\right\vert
;\mathbb{Z}\right)  \rightarrow H_{2n-1}\left(  X-\left\vert \Sigma\right\vert
,\left\vert \Sigma^{\prime}\right\vert ;\mathbb{Z}\right)  \right) \nonumber
\end{gather}
give a three-step weight filtration to $H_{2n-1}\left(  X-\left\vert
\Sigma\right\vert ,\left\vert \Sigma^{\prime}\right\vert ;\mathbb{Z}\right)  $.

\section{Weight filtrations of real vector spaces defined over the integers}

The last part of the previous section, in particular the mapping $\left(
\ref{11}\right)  $, motivates the following real analogue to (a special case
of) the weight filtration in the theory of mixed Hodge structures.

\begin{definition}
Let $M_{\mathbb{Z}}$ be a free $\mathbb{Z}$-module of rank $r+2$ with a fixed
filtration%
\[
M_{\mathbb{Z}}=W_{2}\supsetneqq W_{1}\supsetneqq W_{0}\supsetneqq\left\{
0\right\}
\]
such that $W_{2}/W_{1}$ and $W_{0}$ are free of rank one and $W_{1}/W_{0}$ is
free of rank $r$. An integral structure $\zeta$ on
\[
M_{\mathbb{R}}=M_{\mathbb{Z}}\otimes\mathbb{R}%
\]
is a filtration preserving morphism of $\mathbb{Z}$-modules
\[
M_{\mathbb{Z}}\rightarrow M_{\mathbb{R}}%
\]
such that the induced maps%
\[
\frac{W_{i}}{W_{i-1}}\rightarrow\frac{W_{i}\otimes\mathbb{R}}{W_{i-1}%
\otimes\mathbb{R}}%
\]
are the tautological inclusions obtained by tensoring $\frac{W_{i}}{W_{i-1}}$
with $\mathbb{R}$.
\end{definition}

For example, the inclusion%
\begin{equation}
H_{2n-1}\left(  X-\left\vert \Sigma\right\vert ,\left\vert \Sigma^{\prime
}\right\vert ;\mathbb{Z}\right)  \rightarrow H_{2n-1}\left(  X-\left\vert
\Sigma\right\vert ,\left\vert \Sigma^{\prime}\right\vert ;\mathbb{R}\right)
\label{intstr}%
\end{equation}
with weight filtration $\left(  \ref{18}\right)  $ gives an integral structure
via the isomorphism $\left(  \ref{distiso}\right)  $.

\begin{definition}
Two \textit{integral} \textit{structures} $\zeta$ and $\zeta^{\prime}$ on
$M_{\mathbb{Z}}\otimes\mathbb{R}$ are equivalent if%
\[
\zeta^{\prime}=\zeta\circ\mu
\]
for some filtration-preserving isomorphism%
\[
\mu:M_{\mathbb{Z}}\rightarrow M_{\mathbb{Z}}.
\]

\end{definition}

Notice that the restriction of an integral structure $\zeta$ to $W_{1}$ is
exactly the same thing as a homomorphism%
\[
\beta_{\zeta}:\frac{W_{1}}{W_{0}}\rightarrow W_{0}\otimes\mathbb{R}%
\]
and that equivalent integral structures $\zeta$ and $\zeta^{\prime}$ induce
the same homomorphism%
\[
\beta_{\zeta}=\beta_{\zeta^{\prime}}:\frac{W_{1}}{W_{0}}\rightarrow\frac
{W_{0}\otimes\mathbb{R}}{W_{0}}.
\]
In the same way the data of an integral structure $\zeta$ exactly induces on
$M_{\mathbb{Z}}/W_{0}$ a homomorphism%
\[
\alpha_{\zeta}:\frac{W_{2}}{W_{1}}\rightarrow\frac{W_{1}\otimes\mathbb{R}%
}{W_{0}\otimes\mathbb{R}}%
\]
and equivalent integral structures $\zeta$ and $\zeta^{\prime}$ induce the
same homomorphism%
\[
\alpha_{\zeta}=\alpha_{\zeta^{\prime}}:\frac{W_{2}}{W_{1}}\rightarrow
\frac{W_{1}\otimes\mathbb{R}}{\left(  W_{0}\otimes\mathbb{R}\right)  +W_{1}}.
\]

\begin{theorem}
For the integral structure $\zeta$ given in $\left(  \ref{intstr}\right)  $,%
\begin{align*}
\alpha_{\zeta}  &  =\alpha_{\Sigma^{\prime}}\in J\left(  X\right) \\
\beta_{\zeta}  &  =\beta_{\Sigma}\in J\left(  X\right)  ^{\vee}.
\end{align*}
The set of equivalence classes of integral structures that induce the pair
$\left(  \alpha_{\Sigma^{\prime}},\beta_{\Sigma}\right)  $ is naturally
isomorphic to the fiber of the topological Poincar\'{e} bundle at $\left(
\alpha_{\Sigma^{\prime}},\beta_{\Sigma}\right)  $.

The integral structure $\left(  \ref{intstr}\right)  $ is determined by
$\left(  \alpha_{\Sigma^{\prime}},\beta_{\Sigma}\right)  $ and the topological
height pairing $\left(  \Sigma.\Sigma^{\prime}\right)  $. Thus $\left(
\Sigma.\Sigma^{\prime}\right)  $ gives a distinguished lifting of the point
$\left(  \alpha_{\Sigma^{\prime}},\beta_{\Sigma}\right)  $ to the fiber of the
(topological) Poincar\'{e} bundle at $\left(  \alpha_{\Sigma^{\prime}}%
,\beta_{\Sigma}\right)  $.
\end{theorem}

\begin{proof}
Recall from Definition \ref{PB} that the Poincar\'{e} circle bundle is defined
by the pasting rule%
\[
\left(  \alpha,\beta,r\right)  \mapsto\left(  \alpha+m,\beta+n,r+n\left(
\alpha\right)  +\beta\left(  m\right)  \right)
\]
under the action of $\left(  m,n\right)  \in H_{2n-1}\left(  X;\mathbb{Z}%
\right)  \times H^{2n-1}\left(  X;\mathbb{Z}\right)  $ on $H_{2n-1}\left(
X;\mathbb{R}\right)  \times H^{2n-1}\left(  X;\mathbb{R}\right)  $. Let
$\left\{  h_{j}\right\}  $ be a basis of $\mathcal{H}_{\mathbb{Z}}^{2n-1}$,
the module of harmonic forms with integral periods. That%
\[
\beta_{\zeta}=\beta_{\Sigma}%
\]
derives from the identity $\left(  \ref{zee}\right)  $%
\[%
{\displaystyle\int\nolimits_{\Theta^{\prime}}}
\omega_{\Sigma}=%
{\displaystyle\int\nolimits_{X}}
\psi_{\Sigma}\wedge h\left(  \Theta^{\prime}\right)  -\Gamma\cdot
\Theta^{\prime}.
\]
In terms of the basis%
\[
\omega_{\Sigma},\left\{  h_{j}\right\}  ,\varepsilon_{\Sigma^{\prime}}%
\]
for $M_{\mathbb{R}}^{\vee}$ in $\left(  \ref{distiso}\right)  $ the integral
structure $\left(  \ref{11}\right)  $ is given by the period matrix%
\[
\left(
\begin{array}
[c]{ccc}%
{\displaystyle\int\nolimits_{\partial\Delta_{\Sigma}}}
\omega_{\Sigma} & \left(  0,\ldots,0\right)  & 0\\
\left\{
{\displaystyle\int\nolimits_{\Gamma_{k}}}
\omega_{\Sigma}\right\}  & \left\{
{\displaystyle\int\nolimits_{\Gamma_{k}}}
h_{j}\right\}  &
\begin{array}
[c]{c}%
0\\
\ldots\\
0
\end{array}
\\%
{\displaystyle\int\nolimits_{\Gamma^{\prime}}}
\omega_{\Sigma} &
{\displaystyle\int\nolimits_{\Gamma^{\prime}}}
h_{j} &
{\displaystyle\int\nolimits_{\partial\Gamma^{\prime}}}
\varepsilon_{\Sigma^{\prime}}%
\end{array}
\right)
\]
where the `diagonal' entries in the above matrix are integers, the $\left(
3,2\right)  $-entry is $\alpha_{\Sigma^{\prime}}$, $\left(  2,1\right)
$-entry is $\beta_{\Sigma}$, and the $\left(  3,1\right)  $-entry is $\left(
\Sigma.\Sigma^{\prime}\right)  $. The diffeomorphism group%
\[
\mathrm{Diff}\left(  X,\left\vert \Sigma\right\vert ,\left\vert \Sigma
^{\prime}\right\vert \right)
\]
has a subgroup
\[
\mathrm{Diff}_{0}\left(  X,\left\vert \Sigma\right\vert ,\left\vert
\Sigma^{\prime}\right\vert \right)  :=\mathrm{ker}\left(  \mathrm{Diff}\left(
X,\left\vert \Sigma\right\vert ,\left\vert \Sigma^{\prime}\right\vert \right)
\rightarrow\mathrm{Aut}\left(  H^{2n-1}\left(  X;\mathbb{Z}\right)  \right)
\right)  .
\]
$\mathrm{Diff}_{0}\left(  X,\left\vert \Sigma\right\vert ,\left\vert
\Sigma^{\prime}\right\vert \right)  $ acts on the above matrix in such a way
that the monodromy action associated to path $\gamma$ in $\mathrm{Diff}%
_{0}\left(  X\right)  $ beginning at the identity diffeomorphism and ending in
$\mathrm{Diff}_{0}\left(  X,\left\vert \Sigma\right\vert ,\left\vert
\Sigma^{\prime}\right\vert \right)  $ acts trivially on the `diagonal' entries
and as%
\[
\left(  \alpha_{\Sigma},\beta_{\Sigma^{\prime}},r\right)  \mapsto\left(
\alpha_{\Sigma}+m_{\gamma},\beta_{\Sigma^{\prime}}+n_{\gamma},r+n_{\gamma
}\left(  \alpha_{\Sigma}\right)  +\beta_{\Sigma^{\prime}}\left(  m_{\gamma
}\right)  \right)  .
\]
But modulo $\mathbb{Z}$ the action of $\gamma$ on the entry $\left(
3,1\right)  $ is exactly%
\[%
{\displaystyle\int\nolimits_{\Gamma^{\prime}}}
\omega_{\Sigma}\mapsto%
{\displaystyle\int\nolimits_{\Gamma^{\prime}}}
\omega_{\Sigma}+n_{\gamma}\left(  \alpha_{\Sigma}\right)  +\beta
_{\Sigma^{\prime}}\left(  m_{\gamma}\right)  .
\]

Said otherwise, the action of $\gamma$ takes the integral structure%
\[
\zeta:M_{\mathbb{Z}}\rightarrow M_{\mathbb{R}}%
\]
to the equivalent integral structure%
\[
M_{\mathbb{Z}}\overset{\gamma_{\ast}}{\longrightarrow}M_{\mathbb{Z}}%
\overset{\zeta}{\longrightarrow}M_{\mathbb{R}}.
\]

\end{proof}

\section{Topological Jacobi inversion}

Finally, as in the introduction, consider two primitive classes $\eta
,\eta^{\prime}\in H^{2n}\left(  W;\mathbb{Z}\right)  $ on a complex manifold
$W$. Let $Z,Z^{\prime}\subseteq W$ be real $2n$-manifolds representing their
Poincar\'{e} duals. Let $\mathbb{L}\subset\mathbb{P}$ be a generic line
(Lefschetz pencil) with $\mathbb{L}^{sm}=\mathbb{L}\cap\mathbb{P}^{sm}$ and
let%
\[
\tilde{Z},\tilde{Z}^{\prime}%
\]
be liftings of $Z,Z^{\prime}$ into $X_{\mathbb{L}}$ to manifolds that in
general position with respect to the fibration%
\[
\chi:X_{\mathbb{L}}\rightarrow\mathbb{L}%
\]
and to each other. Thus $\tilde{Z}$ and $\tilde{Z}^{\prime}$ intersect
transversely over $m$ distinct points $p_{j}\in\mathbb{L}^{sm}$. Define%
\begin{align*}
\Sigma_{p}  &  =X_{p}\cap\tilde{Z}\\
\Sigma_{p}^{\prime}  &  =X_{p}\cap\tilde{Z}^{\prime}.
\end{align*}
(Notice that $\Sigma_{p}$ or $\Sigma_{p}^{\prime}$ may be slightly singular,
but only at a finite number of points of $\mathbb{L}^{sm}-\cup_{j}\left\{
x_{j}\right\}  $. These singularities will not affect any of the above
computations since for them $\Sigma_{p}$ and $\Sigma_{p}^{\prime}$ need only
be rectifiable currents at these points.) The computation leading to $\left(
\ref{inf}\right)  $ then shows that the monodromy of the height pairing
computes the local obstruction to separating $\Sigma_{p_{j}}$ and
$\Sigma_{p_{j}}^{\prime}$, that is, the local intersection number of
$\tilde{Z}$ and $\tilde{Z}^{\prime}$.

\part{Analytic setting}

\section{The case in which $\Sigma$ is of type $\left(  n,n\right)  $}

From now on suppose that $\Sigma$ and $\Sigma^{\prime}$ are analytic $\left(
n-1\right)  $-cycles. Although the choice will not matter for the purposes of
what follows, we will endow $J\left(  X\right)  $ with the (Weil) complex
structure given by the operator%
\[
-C:H^{2n-1}\left(  X;\mathbb{R}\right)  ^{\vee}\rightarrow H^{2n-1}\left(
X;\mathbb{R}\right)  ^{\vee}.
\]

In this case, the topological height pairing $\left(  \Sigma.\Sigma^{\prime
}\right)  $ is the real part of a complex-valued function%
\[%
{\displaystyle\int\nolimits_{\Gamma^{\prime}}}
\left(  d^{\ast}+i\left(  d^{C}\right)  ^{\ast}\right)  G\left(
\Sigma\right)  =\frac{1}{2}%
{\displaystyle\int\nolimits_{\Gamma^{\prime}}}
\overline{\partial}^{\ast}G\left(  \Sigma\right)  .
\]
Referring to $\left(  \ref{anal}\right)  $ the imaginary part of this function
is the classical height pairing as defined by Beilinson, Bolch and
Gillet-Soul\'{e}. We denote this last pairing as $\left[  \Sigma
.\Sigma^{\prime}\right]  $. Finally we call%
\[
\left\langle \Sigma.\Sigma^{\prime}\right\rangle :=\left(  \Sigma
.\Sigma^{\prime}\right)  +i\left[  \Sigma.\Sigma^{\prime}\right]
\]
the \textit{holomorphic height pairing} for reasons explained by the next Lemma.

Suppose now that%
\[
\Sigma^{\prime}\left(  t\right)  =q_{\ast}p^{\ast}\left(  \left\{  t\right\}
\right)  -q_{\ast}p^{\ast}\left(  \left\{  t_{0}\right\}  \right)
\]
for a flat family%
\[%
\begin{array}
[c]{ccc}%
I_{T}\subseteq T\times X & \overset{q}{\longrightarrow} & X\\
\downarrow^{p} &  & \\
T &  &
\end{array}
\]
of effective analytic $\left(  n-1\right)  $-cycles where $T$ is a smooth,
connected projective manifold and $t,t_{0}\in T$. In what follows we will want
to handle the case in which $\mathrm{dim}\left(  \left\vert \Sigma\right\vert
\cap\left\vert \Sigma^{\prime}\left(  t\right)  \right\vert \right)  >0$ for
some $t\in T$ and still be able to compute an incidence divisor for $\Sigma$
on $T$, at least up to rational equivalence. We do this by using \cite{FL} to
deform $\Sigma$ in its rational equivalence class to a cycle $\hat{\Sigma}$
such that
\[
\left\{  t:\mathrm{dim}\left(  \left\vert \hat{\Sigma}\right\vert
\cap\left\vert \Sigma^{\prime}\left(  t\right)  \right\vert \right)
>0\right\}
\]
is of codimension at least $2$ in $T$. Then the incidence divisor%
\[
p_{\ast}\left(  q^{\ast}\left(  \hat{\Sigma}\right)  \cdot I_{T}\right)
\]
is well-defined and its rational equivalence class is independent of the
choice of rationally equivalent deformation $\hat{\Sigma}$.

To check that%
\[
\alpha\left(  \Sigma\left(  t\right)  -\Sigma\left(  t_{0}\right)  \right)
\]
is holomorphic in $t$ at a general element $t\in T$, we may suppose that
$\Sigma\left(  t\right)  $ is smooth and irreducible there and that $T$ is a
disk of complex dimension one. So, for any $h^{\prime}\in\mathcal{H}%
^{2n-1}=\mathcal{H}^{2n-1}\left(  X;\mathbb{R}\right)  $,%
\[
\left.  h^{\prime}\right\vert _{I_{T}}=\overline{h^{n-1,n}}+h^{n-1,n}%
\]
has pure type $\left(  n,n-1\right)  +\left(  n-1,n\right)  $. Since $I_{T}$
is locally the product of $T$ with an open set on $\Sigma\left(  t\right)  $
the form $h^{\prime}$ can be written as
\[
\omega\wedge dt+\omega\wedge\overline{dt}+\beta\wedge dt\wedge\overline{dt}%
\]
with $\omega$ real of type $\left(  n-1,n-1\right)  $. Then, by the Cartan-Lie
formula,
\begin{align*}
\overline{\frac{\partial}{\partial t}}%
{\displaystyle\int\nolimits_{\Gamma\left(  t\right)  }}
\left(  h^{\prime}-iCh^{\prime}\right)   &  =%
{\displaystyle\int\nolimits_{\Sigma\left(  t\right)  }}
\left\langle \left.  \overline{\frac{\partial}{\partial t}}\right\vert \left(
h^{\prime}-iCh^{\prime}\right)  \right\rangle \\
&  =%
{\displaystyle\int\nolimits_{\Sigma\left(  t\right)  }}
\omega-\omega=0.
\end{align*}
In particular we have the following.

So in particular, the image of $CH_{alg}^{n}\left(  X\right)  $ is a complex
subtorus $A_{X}$ of the Weil complex torus $J\left(  X\right)  $ (and of the
Griffiths complex torus $J\left(  X\right)  $).

\begin{lemma}
\label{holo'}For any $h^{\prime}\in\mathcal{H}^{2n-1}$, the image%
\[
p_{\ast}\left(  \left.  q^{\ast}\left(  h^{\prime}\right)  \right\vert
_{I_{T}}\right)
\]
is the sum of a globally defined holomorphic one-form on $T$ and its conjugate.
\end{lemma}

We will also need an extension of that standard fact about the holomorphicity
of the Abel-Jacobi map.

\begin{lemma}
\label{holo}The differential%
\[
d%
{\displaystyle\int\nolimits_{\Sigma^{\prime}\left(  t_{0}\right)  }%
^{\Sigma^{\prime}\left(  t\right)  }}
\overline{\partial}^{\ast}G\left(  \Sigma\right)
\]
is holomorphic on $\left(  T-\left\vert p_{\ast}q^{\ast}\left(  \Sigma\right)
\right\vert \right)  $ with logarithmic poles and integral residues at
components of $\left\vert p_{\ast}q^{\ast}\left(  \Sigma\right)  \right\vert $.
\end{lemma}

\begin{proof}
It will suffice to prove the Lemma in the case that $T$ is an algebraic
curve.
\begin{align*}
&
{\displaystyle\int\nolimits_{\Sigma^{\prime}\left(  t_{0}\right)  }%
^{\Sigma^{\prime}\left(  t\right)  }}
\left(  d^{\ast}+i\left(  d^{C}\right)  ^{\ast}\right)  G\left(  \Sigma\right)
\\
&  =\frac{1}{2}%
{\displaystyle\int\nolimits_{\Sigma^{\prime}\left(  t_{0}\right)  }%
^{\Sigma^{\prime}\left(  t\right)  }}
\overline{\partial}^{\ast}G\left(  \Sigma\right)  .
\end{align*}
Now $\overline{\partial}^{\ast}G\left(  \Sigma\right)  $ is of type $\left(
n,n-1\right)  $ and is $d$- and $\overline{\partial}$-closed off $\left\vert
\Sigma\right\vert $. So by the Cartan-Lie formula%
\begin{align*}
d%
{\displaystyle\int\nolimits_{\Sigma^{\prime}\left(  t_{0}\right)  }%
^{\Sigma^{\prime}\left(  t\right)  }}
\overline{\partial}^{\ast}G\left(  \Sigma\right)   &  =\left(
{\displaystyle\int\nolimits_{\Sigma^{\prime}\left(  t\right)  }}
\left\langle \left.  \frac{\partial}{\partial t}\right\vert \overline
{\partial}^{\ast}G\left(  \Sigma\right)  \right\rangle \right)  dt+\left(
{\displaystyle\int\nolimits_{\Sigma^{\prime}\left(  t\right)  }}
\left\langle \left.  \overline{\frac{\partial}{\partial t}}\right\vert
\overline{\partial}^{\ast}G\left(  \Sigma\right)  \right\rangle \right)
\overline{dt}\\
&  =\left(
{\displaystyle\int\nolimits_{\Sigma^{\prime}\left(  t\right)  }}
\left\langle \left.  \frac{\partial}{\partial t}\right\vert \overline
{\partial}^{\ast}G\left(  \Sigma\right)  \right\rangle \right)  dt
\end{align*}
by type. So $\left\langle \Sigma.\Sigma^{\prime}\left(  t\right)
\right\rangle $ is a (multivalued) holomorphic function with real periods.

Now by \cite{Le} or \cite{Wa} formula $\left(  1.1\right)  $, $\left[
\Sigma.\left(  \Sigma^{\prime}\left(  t\right)  \right)  \right]  $ has
logarithmic growth as $t$ approaches a point $t_{1}\in\left\vert p_{\ast
}q^{\ast}\left(  \Sigma\right)  \right\vert $. Alternatively, following
\cite{GS}, \S 1.3, there is a log resolution $Z$ of $X$ with center
$\left\vert \Sigma\right\vert $ on which the Green current of log type
$\eta_{\Sigma}$ with $d^{\ast}G\left(  \Sigma\right)  =d^{C}\eta_{\Sigma}$ can
be written locally in the form
\[
\eta_{\Sigma}\sim%
{\displaystyle\sum\nolimits_{i}}
a_{i}\mathrm{log}\left(  s_{i}\overline{s_{i}}\right)  +\beta
\]
where the $a_{i}$ and $\beta$ are smooth forms (\cite{GS}, (1.3.2.1)). By
further blow-ups supported over $\left\vert \Sigma\right\vert $ , we can
assume that $Z$ is so constructed that the proper transform $\hat{I}_{T}$ of
$q\left(  I_{T}\right)  $ meets the exceptional simple normal-crossing divisor
transversely along a central fiber $\widetilde{\Sigma_{t_{1}}^{\prime}}$. Then
the inequality%
\[
C\left\vert \mathrm{log}\left(  t-t_{1}\right)  \overline{\left(
t-t_{1}\right)  }\right\vert +C^{\prime}\geq\left.
{\displaystyle\sum\nolimits_{i}}
\left\vert \mathrm{log}\left(  s_{i}\overline{s_{i}}\right)  \right\vert
\right\vert _{\hat{I}_{T}}%
\]
for some positive constants $C$ and $C^{\prime}$ reduces the logarithmic
growth assertion to a calculus exercise.

Thus $\left[  \Sigma.\left(  \Sigma^{\prime}\left(  t\right)  \right)
\right]  $ is a well-defined, harmonic function of $t$ on $\left(
T-\left\vert p_{\ast}q^{\ast}\left(  \Sigma\right)  \right\vert \right)  $ and
has logarithmic growth at $\left\vert p_{\ast}q^{\ast}\left(  \Sigma\right)
\right\vert $, and so by an elementary theorem in one complex variable (e.g.
\cite{A}, p. 165, ex. 2) we can locally write%
\[
\left[  \Sigma.\left(  \Sigma^{\prime}\left(  t\right)  \right)  \right]
=\lambda\cdot\mathrm{log}\left\vert t-t_{1}\right\vert +u\left(  t\right)
\]
where $u$ is harmonic on a neighborhood of $t_{1}$. Thus locally%
\begin{equation}
\left\langle \Sigma.\left(  \Sigma^{\prime}\left(  t\right)  \right)
\right\rangle =\lambda\cdot2\pi i\cdot\mathrm{log}\left(  t-t_{1}\right)
+v\left(  t\right)  \label{lochol}%
\end{equation}
where $v$ is holomorphic in a neighborhood of $t_{1}$.
\end{proof}

For example, suppose $X=T$ is a smooth projective curve and $q_{\ast}p^{\ast
}\left(  \left\{  t\right\}  \right)  =\left\{  t\right\}  $. Let $\psi_{t}$
be the unique meromorphic $1$-form on $T$ with real periods and logarithmic
poles at $t$ and $t_{0}$ and residue $1$ at $t$ and residue $-1$ at $t_{0}$
and let $\gamma_{t}$ be a path on $T$ from $t_{0}$ to $t$. Then Lemma
\ref{holo} just above implies that%
\[
\frac{1}{2}\overline{\partial}^{\ast}G\left(  \left\{  t\right\}  -\left\{
t_{0}\right\}  \right)  =\psi_{t}.
\]
Then%
\[
t^{\prime}\mapsto e^{2\pi i%
{\displaystyle\int\nolimits_{t_{0}}^{t^{\prime}}}
\psi_{t}}%
\]
is a meromorphic section of the flat line bundle whose $U\left(  1\right)  $
representation is given by the periods of $\psi_{t}$. The divisor given by
this section is $\left\{  t\right\}  -\left\{  t_{0}\right\}  $.

We continue to suppose that $\Sigma^{\prime}\left(  t\right)  $ is given by a
flat family%
\[
\left(  q_{\ast}p^{\ast}\left(  \left\{  t\right\}  \right)  -q_{\ast}p^{\ast
}\left(  \left\{  t_{0}\right\}  \right)  \right)
\]
as above. $\left[  \Sigma.\Sigma^{\prime}\left(  t\right)  \right]  $ is
well-defined on $\left(  T-\left\vert p_{\ast}q^{\ast}\left(  \Sigma\right)
\right\vert \right)  $ and has logarithmic growth at $\left\vert p_{\ast
}q^{\ast}\left(  \Sigma\right)  \right\vert $. $\left(  \Sigma.\left(
\Sigma^{\prime}\left(  t\right)  -\Sigma^{\prime}\left(  t_{0}\right)
\right)  \right)  $ is locally well-defined on $\left(  T-\left\vert p_{\ast
}q^{\ast}\left(  \Sigma\right)  \right\vert \right)  $ and has integral
periods around $\left\vert p_{\ast}q^{\ast}\left(  \Sigma\right)  \right\vert
$ since%
\[
d\left(  d^{\ast}G\left(  \Sigma\right)  \right)  =\Sigma
\]
and $\Sigma$ is an integral cycle.

\begin{lemma}
\label{fair}Let $\Sigma\in CH_{hom}^{n}\left(  X\right)  $ and let%
\[
D_{\Sigma}:=p_{\ast}\left(  q^{\ast}\left(  \Sigma\right)  \cdot I_{T}\right)
.
\]
Then%
\begin{equation}
D_{\Sigma}=\left(  e^{2\pi i\left\langle \Sigma.\left(  \Sigma^{\prime}\left(
t\right)  -\Sigma^{\prime}\left(  t_{0}\right)  \right)  \right\rangle
}\right)  . \label{interper}%
\end{equation}

\end{lemma}

\begin{proof}
Again it suffices to consider the case in which $T$ is a smooth curve. If
$\Delta\subseteq T$ is a small disk around a point of $\left\vert p_{\ast
}\left(  q^{\ast}\left(  \Sigma\right)  \cdot I_{T}\right)  \right\vert $,
then with respect to the family%
\[%
\begin{array}
[c]{ccc}%
I_{\Delta}\subseteq\Delta\times X & \overset{q}{\longrightarrow} & X\\
\downarrow^{p} &  & \\
\Delta &  &
\end{array}
\]
we have%
\begin{align*}
\mathrm{deg}\left(  \Delta\cdot D_{\Sigma}\right)   &  =\mathrm{deg}\left(
\Delta\cdot_{\Delta}p_{\ast}\left(  q^{\ast}\left(  \Sigma\right)  \cdot
I_{T}\right)  \right) \\
&  =\mathrm{deg}\left(  p^{\ast}\left(  \Delta\right)  \cdot_{\Delta\times
X}q^{\ast}\left(  \Sigma\right)  \right) \\
&  =\mathrm{deg}\left(  q_{\ast}p^{\ast}\left(  \Delta\right)  \cdot_{X}%
\Sigma\right)  .
\end{align*}
Now use that, if $q\left(  I_{\Delta}\right)  $ and $\Sigma$ intersect
properly at $x_{0}\in X$ and $B_{x_{0}}$ is a ball around $x_{0}$ in $X$, then%
\[%
{\displaystyle\int\nolimits_{\partial B_{x_{0}}\cap q\left(  I_{\Delta
}\right)  }}
d^{\ast}G\left(  \Sigma\right)
\]
equals the intersection multiplicity of $q\left(  I_{\Delta}\right)  $ and
$\Sigma$ at $x_{0}$.
\end{proof}

Furthermore, following Narasimhan and Seshadri, holomorphic line bundles of
degree $0$ on $J\left(  X\right)  $ correspond to unitary representations%
\[
\pi_{1}\left(  J\left(  X\right)  ,0\right)  \rightarrow U\left(  1\right)  .
\]
(See \cite{NS} \S 12. Alternatively see \cite{Mum}, pp. 86-87.) Said otherwise
any element
\begin{align*}
\beta &  \in J\left(  X\right)  ^{\vee}=\mathrm{Hom}\left(  H_{2n-1}\left(
X;\mathbb{Z}\right)  ,\frac{\mathbb{R}}{\mathbb{Z}}\right) \\
&  =\mathrm{Hom}\left(  \pi_{1}\left(  J\left(  X\right)  \right)  ,U\left(
1\right)  \right)
\end{align*}
corresponds to a line bundle $L_{\beta}$ with flat unitary connection.

\begin{corollary}
\label{good}Let $\left\{  \Sigma\right\}  \in CH_{hom}^{n}\left(  X\right)  $.
Let%
\[
L_{\beta_{\Sigma}}%
\]
denote the flat line bundle associated to $\beta_{\Sigma}$. For any family
$\left\{  \Sigma^{\prime}\left(  t\right)  \right\}  _{t\in T}$ of algebraic
$\left(  n-1\right)  $-cycles homologous to zero with $T$ smooth, let%
\begin{gather*}
T\overset{\alpha_{T}}{\longrightarrow}J\left(  X\right) \\
t\mapsto\left(  q_{\ast}p^{\ast}\left(  \gamma_{t}\right)  \cdot I_{T}\right)
\end{gather*}
for $\partial\gamma_{t}=\left\{  t\right\}  -\left\{  t_{0}\right\}  $ denote
the Abel-Jacobi mapping. Then $e^{2\pi i\left\langle \Sigma.\left(
\Sigma^{\prime}\left(  t\right)  \right)  \right\rangle }$ is a meromorpic
section of $\alpha_{T}^{\ast}\left(  L_{\beta_{\Sigma}}\right)  $ and%
\begin{equation}
\alpha_{T}^{\ast}\left(  L_{\beta_{\Sigma}}\right)  =\mathcal{O}_{T}\left(
D_{\Sigma}\right)  . \label{rep}%
\end{equation}

\end{corollary}

\begin{proof}
As in $\left(  \ref{zee}\right)  $ we have modulo $\mathbb{Z}$ that%
\[%
{\displaystyle\int\nolimits_{\Theta^{\prime}}}
d^{\ast}G\left(  \Sigma\right)  \equiv%
{\displaystyle\int\nolimits_{\Gamma}}
h\left(  \Theta^{\prime}\right)
\]
for any $\Theta^{\prime}\in H_{2n-1}\left(  X;\mathbb{Z}\right)  $. In fact,
if $\Theta^{\prime}$ runs through an integral basis of $H_{2n-1}\left(
X;\mathbb{Z}\right)  $, $\left\{
{\displaystyle\int\nolimits_{X}}
h_{\Gamma}\wedge h_{\Theta^{\prime}}\right\}  $ exactly describes the image of
$\Sigma$ under the
\[
CH_{hom}^{n}\left(  X\right)  \overset{\alpha}{\longrightarrow}J\left(
X\right)
\]
while $%
{\displaystyle\int\nolimits_{\Theta^{\prime}}}
d^{\ast}G\left(  \Sigma\right)  $ exactly determines the periods of $d^{\ast
}G\left(  \Sigma\right)  $ and therefore the $U\left(  1\right)
$-representation%
\[
\pi_{1}\left(  J\left(  X\right)  ,0\right)  \rightarrow U\left(  1\right)
\]
defining $L_{\beta_{\Sigma}}$. Since%
\[%
{\displaystyle\int\nolimits_{\Theta^{\prime}}}
d^{\ast}G\left(  \Sigma\right)  \equiv%
{\displaystyle\int\nolimits_{X}}
h_{\Gamma}\wedge h_{\Theta^{\prime}}%
\]
for $\Theta^{\prime}\in H_{2n-1}\left(  X;\mathbb{Z}\right)  $, we have by
Lemma \ref{fair} that the meromorphic section $\left(  e^{2\pi i\left\langle
\Sigma.\left(  \Sigma^{\prime}\left(  t\right)  -\Sigma^{\prime}\left(
t_{0}\right)  \right)  \right\rangle }\right)  $ of $\alpha_{T}^{\ast}\left(
L_{\alpha\left(  \Sigma\right)  }\right)  $ has associated divisor $D_{\Sigma
}$.
\end{proof}

Notice that, if $X$ is a curve, $T=X=I_{T}$, and $\Delta\in CH_{0}\left(
X\right)  _{hom}$, then $\Delta=\Sigma\left(  \Delta\right)  =D_{\Delta}$ and
the content of the above corollary is Abel's theorem.

\end{document}